\documentclass{article}
\usepackage[a4paper, total={7in, 10in}]{geometry}
\usepackage{amsmath,amsthm,amssymb,mathrsfs,amstext, titlesec,enumitem, comment, graphicx, color, xcolor, inputenc}
\usepackage{hyperref}
\usepackage{xpatch}
\makeatletter
\def\Ddots{\mathinner{\mkern1mu\raise\p@
\vbox{\kern7\p@\hbox{.}}\mkern2mu
\raise4\p@\hbox{.}\mkern2mu\raise7\p@\hbox{.}\mkern1mu}}
\makeatother

\titleformat{\section}
{\normalfont\Large\bfseries}{\thesection}{1em}{}
\titleformat*{\subsection}{\Large\bfseries}
\titleformat*{\subsubsection}{\large\bfseries}
\titleformat*{\paragraph}{\large\bfseries}
\titleformat*{\subparagraph}{\large\bfseries}

\makeatother
\theoremstyle{plain}
\newtheorem{thm}{Theorem}[section]

\newtheorem{cor}[thm]{Corollary}

\theoremstyle{definition}
\newtheorem{defn}[thm]{Definition}

\newcommand{\thistheoremname}{}
\newtheorem*{genericthm*}{\thistheoremname}
\newenvironment{namedthm*}[1]
  {\renewcommand{\thistheoremname}{#1}%
   \begin{genericthm*}}
  {\end{genericthm*}}

\newcommand{\N}{\mathbb{N}}
\newcommand{\Z}{\mathbb{Z}}

\newcommand{\F}{\mathcal{F}}

\newcommand{\p}{\mathcal{P}}
\newcommand{\bN}{\beta\mathbb{N}}

\newcommand{\FS}{\operatorname{FS}}

\date{\vspace{-5ex}}

\title{\textbf{On Polynomial Extensions of van der Waerden's Theorem and its Applications}}
\author{Sayan Goswami\\  \textit{sayan92m@gmail.com}\footnote{Ramakrishna Mission Vivekananda Educational and Research Institute, Belur, Howrah, 711202, India}}

\begin{document}

\maketitle
\begin{abstract}
In this article, we investigate polynomial generalizations of the van der Waerden theorem with a focus on largeness properties of recurrence patterns. We prove an $IP_r^\star$-strengthened version of the polynomial van der Waerden theorem, where the recurrence set is guaranteed to be large in a precise combinatorial sense.
As applications, we obtain new monochromatic polynomial configurations in both additive and multiplicative settings, including refined results over sum subsystems of IP-sets. Additionally, we prove exponential monochromatic patterns are abundant.
\end{abstract}


\vspace{1mm}


\section{Introduction}

Arithmetic Ramsey theory deals with the monochromatic patterns found in any given finite coloring of the
integers or of the natural numbers $\N$. Here, ``coloring” means disjoint partition, and a set is called ``monochromatic” if it is included in one piece of the partition. Let $\F$ be a family of finite subsets of $\N.$ If for every finite coloring of $\N$, there exists a monochromatic member of $\F$, then such a family $\F$ is called a partition regular family. So basically, Ramsey theory is the study of the classifications of partitioned regular families. Arguably the first substantial development in this area of research was due to I. Schur \cite{sc} in $1916,$ when he proved that the family $\{\{x,y,x+y\}:x\neq y\}$ is a partitioned regular family over $\N.$
The second cornerstone development was due to Van der Waerden in $1927$ when he proved the following theorem.
\begin{thm}[\textbf{Van der Waerden theorem}, \cite{14}]\label{ Van der Waerden}
For any finite coloring of the natural numbers one always finds arbitrarily long monochromatic arithmetic progressions. In other words, the set of all arithmetic progressions of finite length is partitioned regularly.
\end{thm}
Nonlinear versions of classical Ramsey theoretic are difficult to prove. One of the first progress in this direction was due to V. Bergelson and A. Liebman in \cite{2}, where they proved the nonlinear version of van der Waerden's theorem, known as Polynomial van der Waerden's Theorem. The authors proved a generalized version of the much stronger Sz\'emeredi's Theorem, but we are not going to discuss it in this paper. They used methods of topological dynamics to prove their results.
Throughout this article, we assume $\mathbb{P}$ to be the set of all polynomials from $\Z$ to $\Z$ with no constant term.

\begin{thm} [\textbf{Polynomial van der Waerden Theorem}, \cite{2}]\label{pvdw}
Let $r\in \mathbb{N},$ and $\mathbb{N}=\bigcup_{i=1}^r C_i$ be a $r$-coloring of $\mathbb{N}$. Then for any finite collection of polynomials $F$ in $\mathbb{P}$, there exist $a,d\in \mathbb{N}$ and $1\leq j\leq r$ such that
$\{a+p(d):p\in F\}\subset C_j.$     
\end{thm}

\subsection{Hindman Theorem}
After Schur proved his result, an immediate question appeared: Does any infinitary extension exist of Schur's theorem? This was a conjecture of R. Graham and B. Rothschild until $1974,$ when N. Hindman \cite{finitesum} solved this conjecture. Before we state his theorem, we need some technical definitions.

\begin{defn}[\textbf{$IP$-set}]\label{ipset}
    Let $(S,+)$ be any commutative semigroup. Define

    \begin{enumerate}
        \item for any nonempty set $X,$ let $\p_f(X)=\{A\subseteq X:|A|<\infty\},$ and
        \item for any injective sequence $\langle x_n\rangle_n,$ define $\FS(\langle x_n\rangle_n)=\left\lbrace \sum_{t\in H}x_t:H\in \p_f(\N)\right\rbrace,$
        \item for any injective sequence $\langle x_n\rangle_n,$ and for $\alpha \in\mathcal{P}_{f}\left(\mathbb{N}\right),$ we write $x_\alpha =\sum_{n\in \alpha}x_{n}$,
        \item a set $A\subseteq S$ is said to be an $IP$ set if there exists an injective sequence $\langle x_n\rangle_n,$ such that $A=\FS(\langle x_n\rangle_n),$
        \item for any $r\in \N$, a set $A\subseteq S$ is said to be an $IP_r$ set if there exists a sequence $\langle x_n\rangle_{n=1}^r$ such that $A=\FS(\langle x_n\rangle_{n=1}^r)=\{\sum_{t\in H}x_t:(\neq \emptyset )H\subseteq \{1,2,\ldots ,r\}\}.$
    \end{enumerate}
\end{defn}

Our work uses the terminology $IP/IP_r$ set to denote the corresponding set in $(\N,+)$. A set $A$ is said to be an $IP^\star$ (resp.$IP_r^\star$ set) if $A$ intersects every $IP$ set ($IP_r$ set).
The following theorem is known as the Hindman Theorem.

\begin{thm}[\textbf{Hindman Theorem}, \cite{finitesum}]
    For every finite coloring of $\N$, there exists a monochromatic $IP$ set.
\end{thm}

\subsection{Polynomial van der Waerden Theorem}

In \cite{2}, V. Bergelson and A. Leibman proved a relatively stronger version of the polynomial van der Waerden theorem. Later, M. Walter \cite{9} found a simple combinatorial proof that uses color-focusing arguments. Then, in \cite{30}, N. Hindman found an algebraic proof.
Before we address our main result, we need some technical terminologies.

For any commutative semigroup $(S,+),$ and $A\subseteq S,$ and $x\in S,$ let $-x+A=\{y:x+y\in A\}.$ A set $A$ is a  thick set if for any finite subset $F\subset S$, there exists an element $x\in S$ such that $F+x=\{f+x:f\in F\}\subset A.$ A set $A$ is a syndetic set if there exists a finite set $F\subset S$ such that $S=\bigcup_{x\in F}-x+A$, where $-x+A=\{y:x+y\in A\}$. A set $A$ is a piecewise syndetic set if there exists a finite set $F\subset S$ such that $\bigcup_{x\in F}-x+A$ is a thick set. Note that Piecewise syndetic sets are partitioned regular.

In this article, we prove the following stronger version of the Polynomial van der Waerden theorem, and then we study some of its applications.
\begin{thm}[\textbf{$IP^\star$  Polynomial van der Waerden Theorem}, \cite{2,30}] \label{ip*} 
Let $A\subseteq\mathbb{Z}$ be a 
piecewise syndetic set, and $F\in \mathcal{P}_f(\mathbb{P})$. Then 
\[
R=\left\{ n:\left\{ m:\left\{ m,m+p\left(n\right):p\in F\right\} \subset A\right\} \text{ is piecewise syndetic }\right\} \text{}
\]
 is an $IP^{\star}$ set.
\end{thm}

In this article, we prove a strengthening of Theorem \ref{ip*}, and then study some of its applications. 

The following theorem recently has been proved in \cite{x} by R. Xiao, using the methods from the Topological dynamics. But here our proof is different.

\begin{thm}[\textbf{$IP_r^\star$  polynomial van der Waerden theorem}] \label{ipr*} 
Let $A\subseteq\mathbb{N}$ be an  piecewise syndetic set 
and $F\in \mathcal{P}_f(\mathbb{P})$. Then there exists $r\in \N$ such that 
\[
\left\{ n:\left\{ m:\left\{ m,m+p\left(n\right):p\in F\right\} \subset A\right\} \text{ is piecewise syndetic }\right\} 
\]
 is an $IP_{r}^{\star}$ set.
\end{thm}

Then, applying this theorem, we prove several new monochromatic structures involving sum subsystems generalized to the polynomial settings.

\section{Preliminaries} 

\subsection{Preliminaries on Ultrafilters}
 
For a set $S$ let $\beta S$ be the set of all ultrafilters on $S$. For  
$s\in S$ by $s$ we mean the principal ultrafilter containing $s.$ For details, we refer to the  \cite{5} for the readers. 
If $(S,\cdot)$ is a semigroup, we can extend the operation ``$\cdot$''. 
 to a semigroup operation on $ \beta S$ by  
\begin{equation}\label{multdef} 
  A \in p \cdot q  \Leftrightarrow \ \{s\in S : s^{-1} A \in q\}\in p.       
\footnote{$S$ is a semigroup, so $s$ might not have an inverse. We may 
avoid this obstacle by defining $s^{-1} A:= \{t\in S: st\in A\}.$} 
\end{equation} 
With this operation, $\beta S$ becomes the Stone-\v{C}ech compactification of $S.$
Applications of the algebraic structure of $\beta S$ in partition Ramsey Theory  
are abundant. Examples are simple proofs of the theorems of Hindman and  
van der Waerden. Idempotent ultrafilters (i.e. ultrafilters $p\in \beta S$ 
satisfying $p\cdot p =p$) are connected with $IP$ sets: it can be shown that a set $A\subseteq S$ is $IP$ iff it is a member of an idempotent ultrafilter. Note for any topological semigroups $T,$ by $E(T)$ we denote the set of all idempotents of $T.$

$\beta S$ is always the smallest (two-sided) ideal which will be denoted by  
$K(\beta S)$. It turns out that for $(S,\cdot)=(\N, +)$ the elements of 
$K(\beta \N,+)$  are well suited for van der Waerden's Theorem. Idempotents in $K(\beta S,\cdot )$ (which are always present) are called 
\emph{minimal idempotents}. 
Not at all surprisingly minimal idempotents are   
particularly interesting for combinatorial applications. Subsets of 
$S$ which are contained in some minimal idempotent are called \emph{central sets}. On the other hand, a set $A\subseteq S$ is called \emph{piecewise syndetic set} if it is a member of an ultrafilter that belongs to $K(\beta S,\cdot).$ If $(S,+)$ is a commutative semigroup, then the notion of a piecewise syndetic set coincides with the definition that we mentioned in the introduction.

\subsection{Preliminaries on Polynomial Hales-Jewett Theorem}

Now we pause to recall the Hales-Jewett Theorem and it's polynomial extension. 
Let $\omega=\mathbb{N}\cup\left\{ 0\right\} $, where $\mathbb{N}$
is the set of positive integers. Given a nonempty set $\mathbb{A}$
called alphabet, a finite word is an expression of the form $w=a_{1}a_{2}\ldots a_{n}$
with $n\geq1$ and $a_{i}\in\mathbb{A}$. The quantity $n$ is called
the length of $w$ and denoted $\left|w\right|$. Let $v$ (a variable)
be a letter not belonging to $\mathbb{A}$. By a variable word over
$\mathbb{A}$ we mean a word $w$ over $\mathbb{A}\cup\left\{ v\right\} $
that has at least one occurrence of $v$. For any variable word $w$,
$w\left(a\right)$ is the result of replacing each occurrence of $v$
by $a$.

The following theorem is known as Hales-Jewett theorem, is due to
A. W. Hales and R. I. Jewett.
\begin{thm}
\textup{\cite[Hales-Jewett Theorem (1963)]{7}} \label{hj222}
For all values $t,r\in\mathbb{N}$, there exists a number $\text{HJ}\left(r,t\right)$
such that, if $N\geq\text{HJ}\left(r,t\right)$ and $\left[t\right]^{N}$
is $r$ colored then there exists a variable word $w$ such that $\left\{ w\left(a\right):a\in\left[t\right]\right\} $
is monochromatic.
\end{thm}

The word space $\left[t\right]^{N}$ is called Hales-Jewett space
or H-J space. The number $\text{HJ}\left(r,t\right)$ is called Hales-Jewett
number.

Before we recall the polynomial Hales-Jewett theorem we need some new notions.
For $q,N\in\mathbb{N}$, let $Q=[q]^{N}$. For $a\in Q$, $\emptyset\neq\gamma\subseteq[N]$
and $1\leq x\leq q$, $a\oplus x\gamma$ is defined to be the vector
$b$ in $Q$ obtained by setting $b_{i}=x$ if $i\in\gamma$ and $b_{i}=a_{i}$
otherwise.

In the statement of Theorem \cite[Polynomial Hales-Jewett Theorem]{9}, we have $a\in Q$ so that
$a=\langle\vec{a}_{1},\vec{a}_{2},\ldots,\vec{a}_{d}\rangle$ where
for $j\in\{1,2,\ldots d\}$, $\vec{a}_{j}\in[q]^{N^{j}}$ and we have
$\gamma\subseteq[N]=\{1,2,\ldots,N\}$. Given $j\in\{1,2,\ldots,d\}$,
let $\vec{a}_{j}=\langle a_{j,\vec{i}}\rangle_{\vec{i}\in N^{j}}$.
Then $a\oplus x_{1}\gamma\oplus x_{2}(\gamma\times\gamma)\oplus\ldots\oplus x_{d}\gamma^{d}=b$
where $b=\langle\vec{b}_{1},\vec{b}_{2},\ldots,\vec{b}_{d}\rangle$
and for $j\in\{1,2,\ldots,d\}$, $\vec{b}_{j}=\langle b_{j,\vec{i}}\rangle_{\vec{i}\in N^{i}}$
where 
\[
b_{j,\vec{i}}=\left\{ \begin{array}{rl}
x_{j} & \hbox{if}\,\vec{i}\in\gamma^{i}\\
a_{j,\vec{i}} & \hbox{otherwise.}
\end{array}\right.
\]

The following theorem is the Polynomial Hales-Jewett Theorem.
\begin{thm}
\textup{\label{PHJ} \cite[Polynomial Hales-Jewett Theorem]{9}}
For any $q,k,d$ there exists $N\left(q,k,d\right)\in\mathbb{N}$
such that whenever $Q=Q\left(N\right)=\left[q\right]^{N}\times\left[q\right]^{N\times N}\times\cdots\times\left[q\right]^{N^{d}}$
is $k$-colored there exist $a\in Q$ and $\gamma\subseteq\left[N\right]$
such that the set of points 
\[
\left\{ a\oplus x_{1}\gamma\oplus x_{2}\left(\gamma\times\gamma\right)\oplus\cdots\oplus x_{d}\gamma^{d}:1\leq x_{i}\leq q\right\} 
\]
 is monochromatic.
\end{thm}

\section{Our Results}

\subsection{ $IP_r^\star$ Polynomial van der Waerden's theorem}\label{p}

As promised above, in this section we  strengthen Theorem \ref{ip*}. We provide both algebraic and combinatorial proofs. Both of these proofs have their own pros and cons. The Algebraic proof is short but uses Zorn's lemma, whereas the combinatorial proof is lengthy, a little complicated, but does not use Zorn's lemma.

\begin{proof}[Proof of Theorem \ref{ipr*}]
\textbf{(Combinatorial proof:)} As $A$ is a piecewise syndetic set, there exists a finite set $F_1\subset\mathbb{N}$
such that $B=\cup_{t\in F_1}\left(-t+A\right)$ is a thick set. 
Assume that $|F_1|=k.$
Let $\mathbb{A}$ be the set of all coefficients of the polynomials
of $F,$ and $d$ be the maximum degree of the polynomials of $F.$ 

To prove the given set is an $IP_r^\star$ set for some $r\in \N$, we need to show that the given set intersects with every set of the form $\FS(\langle x_{n}\rangle_{n=1}^{r})$ for some $r\in \N$. Our claim is that $r=N.$
Let $N=N\left(|\mathbb{A}|,k,d\right)\in \N$ be the polynomial Hales-Jewett number guaranteed by Theorem \ref{PHJ}. Let
$\langle x_{n}\rangle_{n=1}^{N}$ be any sequence in $\mathbb{N}.$
Let $Q=Q\left(N\right)$ be the set as in Theorem \ref{PHJ}. Define the map $\gamma:Q\rightarrow\mathbb{Z}$
by
\[
\gamma\left(\left(a_{1},\ldots,a_{N},a_{11},\ldots,a_{NN},\ldots,a_{11\cdots1(d\text{ times})},\ldots,a_{NN\cdots N(d\text{ times})}\right)\right)=
\]

\[
\sum_{i=1}^{n}a_{i}x_{i}+\sum_{i,j=1,1}^{N,N}a_{ij}x_{i}x_{j}+\cdots+\sum_{i_{1},i_{2},\ldots,i_{d}=1,1,\cdots,1}^{N,N,\cdots,N}a_{i_{1}i_{2}\cdots i_{d}}x_{i_{1}}x_{i_{2}}\cdots x_{i_{d}}.
\]

If necessary, translate the set $G=\left\{ \gamma\left(\overrightarrow{a}\right):\overrightarrow{a}\in Q\right\} $ by a sufficiently large number $M$ to assume that $G\subset \mathbb{N}$. Then for any finite coloring
of $G,$ there exists $a\in \N,$ and $d\in FS\left(\langle x_{n}\rangle_{n=1}^{N}\right)$ such that  $\left\{ a+p\left(d\right):p\in F\right\} \cup\left\{ a\right\}$ is monochromatic.

As $G$ is finite, the number of such patterns in $G$
is finite. Let $\mathscr{H}$ be the collection of such patterns
in $G$.
Let $|\mathscr{H}|=l$, $\mathscr{H}=\left\{ H_{i}:1\leq i\leq l\right\}$ and $F_1=\{f_i:1\leq i\leq k\}.$
As $G$ is finite, the collection $C=\left\{ x:G+x\subset B\right\} $
is a thick set.

Now construct a $[l]\times[k]$ coloring of $C$ by $x\in C$ has
color $\left(i,j\right)\in[l]\times[k]$ if and only if there exists
$f_j\in F_1$ and $H_{i}\in\mathscr{H}$ such that $\left\{ x+f_j+m:m\in H_{i}\right\} \subset A$
(one can choose the least $\left(i,j\right)$ in dictionary ordering).
Now $x\in C_{\left(i,j\right)}$ if and only if $x\in C$ and $x$
has color $\left(i,j\right).$ So, $C=\cup_{\left(i,j\right)\in[l]\times[k]}C_{\left(i,j\right)}.$
Hence at least one of $C_{(i,j)}'s$, say $C_{\left(s,t\right)}$
is a piecewise syndetic set. Hence there exists $H=\left\{ q+p\left(r\right):p\in F\right\} \cup\left\{ q\right\} \in \mathcal{H} $ such that for all $x\in C_{\left(s,t\right)}$
 $$\left\{ x+\left(t+q\right)+p\left(r\right):p\in P\right\} \cup\left\{ x+\left(t+q\right)\right\} \subset A.$$

As $C_{\left(s,t\right)}$
is a piecewise syndetic set, and as piecewise syndetic sets are translation invariant, we have $C_{\left(s,t\right)}+\left(t+q\right)$
is a piecewise syndetic set. This implies for all $m\in C_{\left(s,t\right)}+\left(t+q\right),$ we have $$\left\{ m\right\} \cup\left\{ m+p\left(d\right):p\in F\right\} \subset A,$$
for some $d\in FS\left(\langle x_{n}\rangle_{n=1}^{N}\right).$ Hence
\[
\left\{ n:\left\{ m:\left\{ m,m+p\left(n\right):p\in F\right\} \subset A\right\} \text{ is piecewise syndetic }\right\} 
\]
 is an $IP_{N}^{\star}$. This completes the proof.
\end{proof}

The key idea of the following algebraic proof is to use \cite[Theorem 4.39]{5} to reduce the problem from piecewise syndetic sets to syndetic sets.
\begin{proof}
\textbf{(Algebraic proof:)} Let $A$ be a piecewise syndetic set. Hence there exists
$p\in K\left(\beta\mathbb{N},+\right)$ such that $A\in p.$
Hence by \cite[Theorem 4.39]{5}, the set $S=\left\{ x:-x+A\in p\right\} $
is a syndetic set. Therefore we have a finite set $F\subset \N$ such that $\mathbb{N}=\cup_{t\in F}-t+S.$ 
Assume that $|F|=r.$ Proceeding along the same line of the above proof, let us choose $N=N\left(\mathbb{A},r,d\right)\in \N$,  a sequence $\langle x_{n}\rangle_{n=1}^{N}$ in $\N$, and the corresponding finite set $G\subset \N$.
As $G\subset\cup_{t\in F}-t+S,$
 there exists $(\emptyset\neq)H\subseteq\left\{ 1,2,\ldots,N\right\}$,
and $b\in\mathbb{N}$ such that $\left\{ b,b+f\left(x_{H}\right):f\in F\right\} \subset S.$
In other words  
\[
D=\left\{ d: \text{ there exists } m\text{ such that }\left\{ m,m+f\left(d\right):f\in F\right\} \subset S\right\} 
\]
 is $IP_{N}^{\star}$ set. Now for each $n\in D,$ there exists $m$
such that $\left\{ m,m+f\left(n\right):f\in F\right\} \subset S.$
Hence $$B=A\cap\left(-m+A\right)\cap\bigcap_{f\in F}-\left(m+f\left(n\right)\right)+A\in p.$$
As $B$ is a piecewise syndetic set, $B+m$ is also a piecewise syndetic, and note that for
each $a\in B+m,$ we have $\left\{ a,a+f\left(n\right):f\in F\right\} \subset A.$
This completes the proof.
\end{proof}

It is easy to check that $\overline{E\left(K(\bN,+)\right)}$ is a left ideal of $(\bN,\cdot)$, and so it contains a minimal left ideal of $(\beta \N,\cdot)$. Therefore we can choose $p\in E\left(K(\bN,\cdot)\right)\cap \overline{E\left(K(\bN,+)\right)}.$ By choice of $p,$ each member of $p$ is additive and multiplicative central set. The following corollary immediately follows from Theorem \ref{ipr*}.

\begin{cor}
Let $F\in \mathcal{P}_f(\mathbb{P})$, and
$p\in E\left(K(\bN,\cdot)\right)\cap \overline{E\left(K(\bN,+)\right)}.$ Then for every $A\in p,$
\[
\left\{ d \in A:\left\{ a:\left\{ a,a+p\left(d\right):p\in F\right\} \subset A\right\} \text{ is additive piecewise syndetic }\right\} 
\]
 is both additive and multiplicative piecewise syndetic. 
\end{cor}

\begin{proof}
Given $p\in E\left(K(\bN,\cdot)\right)\cap \overline{E\left(K(\bN,+)\right)}.$
As $A\in p,$ we have  $A$ is an additive and multiplicative piecewise syndetic set. Hence from Theorem \ref{ipr*} we have
\[
D=\left\{ d:\left\{ a:\left\{ a,a+p\left(d\right):p\in F\right\} \subset A\right\} \text{ is additive piecewise syndetic }\right\} 
\]
 is an $IP_{r}^{\star}$ set for some $r\in \N.$ As for each $r\in \N$, every multiplicative piecewise syndetic set contains an $IP_r$ set, we have $\overline{K\left(\beta\mathbb{N},\cdot\right)}\subset\overline{D}.$
So, $D\in p$ and as $A\in p$, we have $A\cap D\in p.$ Hence we have our desired conclusion. 
\end{proof}

\subsection{Applications to Polynomial Patterns over Sum-Subsystems}\label{appl}

In this section, we deduce two new applications of Theorem \ref{ipr*}. The first one involves patterns in additively piecewise syndetic sets, where the domain of the polynomials can be chosen from any given multiplicative central set. And the second one involves the evaluation of a polynomial over the ``sum subsystem" of any given $IP$ set. 
\begin{thm}\label{D-1}
Let $F\in \mathcal P_f(\mathbb{P})$, $A$ be an additively piecewise syndetic set, and $B$ be a multiplicatively central set. Then there exists a sequence $\langle x_n \rangle_{n=1}^\infty$ such that
 \begin{enumerate}
  \item[(1)] $FP(\langle x_n \rangle_{n=1}^\infty)\subset B$, and
  \item[(2)] for each $N\in \mathbb{N}$, there exists $a(N)\in \mathbb{N}$ such that $$\left\lbrace a(N), a(N)+P(y): P\in F, y\in  FS(\langle x_n \rangle_{n=1}^N) \cup FP(\langle x_n \rangle_{n=1}^N)\right\rbrace\subset A.$$
 \end{enumerate}
 In addition for each $N\in \N$, such collection of $a(N)$ is piecewise syndetic.
\end{thm}

\begin{proof}
Given $B$ is a multiplicatively central set.
Hence there exists $q\in E(K(\beta \mathbb{N}, \cdot))$  such that $B\in q$.
By \cite[Lemma 4.14]{5}, we have  $B^*=\{n\in B: n^{-1}B\in q\}\in q$. As $A$ is additively piecewise syndetic set, from Theorem \ref{ipr*} we have $$D= \{d:\{a: \{a, a+P(d): P\in F\}\subset A\} \text { is a piecewise syndetic}\}$$ is an $IP_r^*$ for some $r\in \mathbb{N}$. Hence, $\overline{K(\beta \mathbb{N}, \cdot)}\subset \overline {D}$, and this implies $B^*\cap D\in q$.
Let $x_1\in B^*\cap D$.
Then there exists $a(1)\in \mathbb{N}$ such that $\{a(1), a(1)+P(x_1): P\in F\}\subset A$ and such collection of $a(1)$ is piecewise syndetic. Inductively assume that for some $N\in \mathbb{N}$, there exists a finite sequence $\langle x_n\rangle_{n=1}^N$ such that
\begin{enumerate}
 \item $FP(\langle x_n \rangle_{n=1}^N)\subset B^*$, and
 \item for $1\leq k \leq N$, $\exists$ $a(k)\in \mathbb{N}$ such that $$B_k=\{a: \{a, a+P(y): P\in F, y\in FS(\langle x_n \rangle_{n=1}^k) \cup FP(\langle x_n \rangle_{n=1}^k)\}\subset A\}$$ piecewise syndetic.
\end{enumerate}
Now $B_N$ is a piecewise syndetic set. For each $y\in FP(\langle x_n \rangle_{n=1}^N)$, $z\in FS(\langle x_n \rangle_{n=1}^N)$, and $P\in F$, define new polynomials $P_y, P^z \in \mathbb{P}$ by 

\begin{itemize}
    \item $P_y(n)=P(ny)$, and
    \item $P^z(n)=P(z+n)-P(z)$ $\forall n\in \mathbb{Z}$.
\end{itemize}

  Define $$G=F\cup \{P_y: P\in F, y\in FP(\langle x_n \rangle_{n=1}^N)\} \cup \{P^z: P\in F, z\in FS(\langle x_n \rangle_{n=1}^N)\},$$ and so $G\in \mathcal P_f(\mathbb{P})$.
Again from Theorem \ref{ipr*}, $$D_1=\{d: \{b: \{b, b+f(d): f\in G\}\subset B_N\}\text { is a piecewise syndetic set}\}$$ is an $IP_s^*$ for some $s\in \mathbb{N}$. From \cite[Lemma 4.14]{5}, we know that for each $y\in FP(\langle x_n \rangle_{n=1}^N)$, $y^{-1}B^*\in q.$ Hence we have $$E= D_1\cap \bigcap_{y\in FP(\langle x_n \rangle_{n=1}^N)}y^{-1}B^*\cap B^*\in q.$$
Choose $x_{N+1}\in E$, and so $FP(\langle x_n \rangle_{n=1}^{N+1})\subset B^*$. Again $x_{N+1}\in D_1$ and so  there exists $a(N+1)$ such that $$\left\lbrace a(N+1), a(N+1)+f(x_{N+1}): f\in G\right\rbrace\subset B_N.$$ Again the collection of such $a(N+1)$ is piecewise syndetic. Now,

\begin{enumerate}
    \item  $a(N+1)\in B_N$ implies $ a(N+1)\in A$, and for all $y\in FP(\langle x_n \rangle_{n=1}^N)\cup FS(\langle x_n \rangle_{n=1}^N)$, and $P\in F$ we have $a(N+1)+P(y)\in A$, \hfill{(A)}

 \item  for all $P\in F, y\in FP(\langle x_n \rangle_{n=1}^N)$, $a(N+1)+P_y(x_{N+1})\in A$ implies for all $y\in FP(\langle x_n \rangle_{n=1}^{N+1})\setminus FP(\langle x_n \rangle_{n=1}^N),$ $a(N+1)+P(y)\in A,$ and \hfill{(B)}

 \item for all $P\in F, z\in FS(\langle x_n \rangle_{n=1}^N)$, $a(N+1)+P^z(x_{N+1})\in A$ implies for all $z\in FP(\langle x_n \rangle_{n=1}^{N+1})\setminus FP(\langle x_n \rangle_{n=1}^N)$ we have $a(N+1)+P(z)\in A.$ \hfill{(C)}
\end{enumerate}
Combining (A), (B), and (C), for all $P\in F$, and $y\in FS(\langle x_n \rangle_{n=1}^{N+1}) \cup FP(\langle x_n \rangle_{n=1}^{N+1})$, we have $a(N+1)+P(y)\in A$. 
As the collection of such $a(N+1)$ is piecewise syndetic, 
We proved the induction hypothesis, and this proves our result.

\end{proof}
For any finite coloring of $\N$, the following special case of Theorem \ref{D-1} says that we can choose the set $FP(\langle x_n \rangle_{n=1}^\infty)$ from the same color as of $a(N)+P(y)$.
\begin{cor}\label{imp}
    Let  $p\in E\left(K(\bN,\cdot)\right)\cap \overline{\left(K(\bN,+)\right)}$, and $A\in p.$ Then from Theorem \ref{D-1}, for any $F\in \mathcal P_f(\mathbb{P})$, there exists a sequence $\langle x_n \rangle_{n=1}^\infty$ such that
$FP(\langle x_n \rangle_{n=1}^\infty)\subset A$, and
  for each $N\in \mathbb{N}$, there exists $a(N)\in \mathbb{N}$ such that $$\{a(N), a(N)+P(y): P\in F, y\in FS(\langle x_n \rangle_{n=1}^N) \cup FP(\langle x_n \rangle_{n=1}^N)\}\subset A.$$
  In addition for each $N\in \N$, such collection of $a(N)$ is piecewise syndetic.
 \end{cor}

Sum subsystems of IP sets play important role in Ramsey theory. These systems appear in many Ramsey theoretic patterns to strengthen classical statements.
Here we recall the definition of sum subsystems.

\begin{defn} \textbf{(Sum subsystem)}
    For any $IP$ set  $FS\left(\langle x_{n}\rangle_{n}\right)$, a sum subsystem of  $FS\left(\langle x_{n}\rangle_{n}\right)$ is a set of the form $FS\left(\langle y_{n}\rangle_{n}\right)$, where for each $n\in \N,$ $y_n$ is defined as follows:

\begin{enumerate}
 \item there exists a sequence $\langle H_n\rangle_n$  in $\mathcal{P}_f(\N)$ such that for each $n\in \N$, $\max H_n<\min H_{n+1}$, and 
    \item $y_n=\sum_{t\in H_n}x_t$ for all $n\in \N$.
\end{enumerate}
\end{defn}

The following theorem is a variant of Theorem \ref{D-1}, where the polynomials are evaluated over sum-subsystems of a given $IP$ set. 
In the later section we will use this Theorem to prove the first improvement of \cite[Theorem 4.1.]{exp}.


\begin{thm}
\label{ess1} Let $A$ be an additive piecewise syndetic set and $F\in \mathcal{P}_f(\mathbb{P})$. Then for every injective sequence
$\langle x_{n}\rangle_{n}$ there exists a sum subsystem $FS\left(\langle y_{n}\rangle_{n}\right)$
of $FS\left(\langle x_{n}\rangle_{n}\right)$ such that for every
$N\in\mathbb{N},$ there exists $a(N)\in\mathbb{N},$ such that 
$$\left\lbrace a(N)+p\left(y\right): y\in FS\left(\langle y_{n}\rangle_{n=1}^{N}\right)\cup FP\left(\langle y_{n}\rangle_{n=1}^{N}\right)\text{and }
p\in F\right\rbrace \subset A.$$
 In addition for each $N\in \N$, such collection of $a(N)$ is piecewise syndetic.
\end{thm}

\begin{proof}
For any $v\in E\left(\beta\mathbb{N},+\right)$ and any $A\in v,$ let $A^{\star}=\{n\in A:-n+A\in v\}.$ As $v$ is an idempotent, we have  $A^{\star}\in v$.
Let $v$ be an idempotent such that $FS\left(\langle x_{n}\rangle_{n}\right)\in v,$
and so $FS\left(\langle x_{n}\rangle_{n}\right)^{\star}\in v.$ From
Theorem \ref{ip*}, we have $y_{1}\in FS\left(\langle x_{n}\rangle_{n}\right)^{\star}$
such that the set $B_{1}=\left\{ a:\left\{ a,a+p(y_{1}):p\in F\right\} \subset A\right\} $
is piecewise syndetic. So there exists an element $a\in B_{1}$ such
that $\left\{ a,a+p(y_{1}):p\in F\right\} \subset A.$ Inductively assume that for some $M\in\mathbb{N}$,

\begin{enumerate}
    \item   $FS\left(\langle y_{n}\rangle_{n=1}^{M}\right)\subset FS\left(\langle x_{n}\rangle_{n}\right)^{\star},$

\item For each $1\leq N\leq M,$ we have $$\left\lbrace a(N)+p\left(y\right): y\in FS\left(\langle y_{n}\rangle_{n=1}^{N}\right)\cup FP\left(\langle y_{n}\rangle_{n=1}^{N}\right)\text{and }
p\in F\right\rbrace \subset A,$$

\item For each $1\leq N\leq M,$ the set 
\[
B_{N}=\left\{ a:\left\{ a,a+p(y):p\in F,y\in FS\left(\langle y_{n}\rangle_{n=1}^{N}\right)\cup FP\left(\langle y_{n}\rangle_{n=1}^{N}\right) \right\} \subset A\right\} \text{ is piecewise syndetic. }
\]
\end{enumerate}
For each $p\in F,$ $y\in FS\left(\langle y_{n}\rangle_{n=1}^{M}\right),$ and $z\in FP\left(\langle y_{n}\rangle_{n=1}^{M}\right),$
define  new polynomials $p_{y},p'_z:\mathbb{N}\rightarrow\mathbb{N},$ by 
\begin{itemize}
    \item $p_{y}\left(n\right)=p\left(n+y\right)-p\left(y\right);$ and
    \item  $p'_{z}\left(n\right)=p(zn).$ 
\end{itemize}

Choose a
new collection of polynomials $$G=F\cup\left\{ p_{y}:p\in F,y\in FS\left(\langle y_{n}\rangle_{n=1}^{M}\right)\right\}\cup\left\{ p'_{z}:p\in F,z\in FP\left(\langle y_{n}\rangle_{n=1}^{M}\right)\right\} .$$
From Theorem \ref{ip*}, 
\[
D=\left\{ x:\left\{ c:\left\{ c,c+p(x):p\in G\right\} \subset B_{M}\right\} \text{ is piecewise syndetic }\right\} 
\]
 is an $IP^{\star}$ set. Note that if $A$ is an $IP^*$ set, and  $n\in \N,$ then $n^{-1}A$ is again an $IP^*$ set.
 Hence, 
\[
D_1=D\cap \bigcap_{z\in FS\left(\langle y_{n}\rangle_{n=1}^{M}\right)}z^{-1}D\cap FS\left(\langle x_{n}\rangle_{n}\right)^{\star}\cap\bigcap_{y\in FS\left(\langle y_{n}\rangle_{n=1}^{M}\right)}\left(-y+FS\left(\langle x_{n}\rangle_{n}\right)^{\star}\right)\in v.
\]
Choose $y_{M+1}\in D_1$, and an element $c_{1}\in E=\left\{ c:\left\{ c,c+p(y_{M+1}):p\in G\right\} \subset B_{M}\right\} .$
Hence $FS\left(\langle y_{n}\rangle_{n=1}^{M+1}\right)\subset FS\left(\langle x_{n}\rangle_{n}\right)^{\star}.$ Now we complete the induction.

\begin{itemize}

   \item  \textbf{Case 1}: if $z\in FS\left(\langle y_{n}\rangle_{n=1}^{M}\right)\cup FP\left(\langle y_{n}\rangle_{n=1}^{M}\right) \cup\left\{ y_{M+1}\right\} $,
then for all $p\in F,$ $c_{1}+p\left(z\right)\in B_{M}\subset A,$ and

   \item  \textbf{Case 2:} if $z\in FS\left(\langle y_{n}\rangle_{n=1}^{M+1}\right) \setminus \left(FS\left(\langle y_{n}\rangle_{n=1}^{M}\right) \cup\left\{ y_{M+1}\right\}\right) $,
then we have $z=s+y_{M+1}$, for some $s\in FS\left(\langle y_{n}\rangle_{n=1}^{M}\right)$.

So we have $$c_{1}+p_{s}\left(y_{M+1}\right)=c_{1}+p\left(s+y_{M+1}\right)-p\left(s\right)\in B_{M},$$ 

But then $c_{1}+p\left(s+y_{M+1}\right)-p\left(s\right)+p\left(s\right)=c_{1}+p\left(z\right)\in A.$

\item \textbf{Case 3:} if $z\in FP\left(\langle y_{n}\rangle_{n=1}^{M+1}\right)\setminus \left( FP\left(\langle y_{n}\rangle_{n=1}^{M}\right) \cup\left\{ y_{M+1}\right\}\right) $, then we have and $z=ty_{M+1}$ for some $t\in FP\left(\langle y_{n}\rangle_{n=1}^{M}\right).$ Then

$$c_{1}+p'_{t}\left(y_{M+1}\right)=c_{1}+p\left(ty_{M+1}\right)=c_1+p(z)\in B_{M}.$$

\end{itemize}

Now rename $E=B_{M+1}$. Now for each $a\in B_{M+1}\subset A,$ we have
\[
\left\{ a+p\left(z\right):p\in F\text{ and }z\in FS\left(\langle y_{n}\rangle_{n=1}^{M+1}\right)\cup FP\left(\langle y_{n}\rangle_{n=1}^{M+1}\right)\right\} \subset A,
\]
 and the set $B_{M+1}$ is piecewise syndetic. This completes the induction.
\end{proof}

\subsection{Monochromatic Exponential Patterns are Abundant}

In 1916, Schur \cite{sc} proved that the pattern  $\{x, y, x + y: x\neq y\}$ is  partition regular. A natural multiplicative analogue of Schur's theorem asserts that the set $\{x, y, x \cdot y\}$ is also partition regular. In $2012$, Sisto \cite{sisto} proved that for any $2$-coloring of $\mathbb{N}$, there exist distinct numbers $x \ne y$ such that the set $\{x, y, x^y\}$ is monochromatic. He further conjectured that this result should hold for arbitrary finite colorings. This conjecture was later confirmed by Sahasrabudhe \cite{4} in $2016$. Infact he proved a more stronger version: exponential version of finitary Hindman theorem. The infinitary version has been recently proved in \cite{6,exp}.

\begin{thm}[{Exponential Hindman Theorem}]\label{wow}
    For every finite coloring (partition) of $\mathbb{N}$, there exists an injective sequence $\langle x_n \rangle_n$ such that
    \[
    \operatorname{FEP}(\langle x_n \rangle_n) = \left\{ x_{i_n}^{x_{i_{n-1}}^{\cdots^{x_{i_1}}}} : 1 \le i_1 < \cdots < i_{n-1} < i_n,\, n \in \mathbb{N} \right\}
    \]
    is monochromatic.
\end{thm}

In this section, we strengthen the above theorem by showing that for each $i\in \mathbb{N},$ the element $x_i$ can be chosen from sets that are sufficiently rich in combinatorial structure. We show the following theorem is an improved version of Theorem \ref{wow}.

\begin{thm}\label{p1} Let $\langle N_i\rangle_i$ be a sequence in $\N.$ Then for every finite coloring of $\N,$ there exists a sequence of finite sets $\langle G_i\rangle_i$ such that $|G_i|=N_i$ for every $i\in \N,$ such that the pattern
$$\left\lbrace x_{i_n}^{x_{i_{n-1}}^{\cdot^{\cdot^{\cdot^{x_{i_1}}}}}}:x_j\in FP(G_j)\, \forall j\in \N, \text{ and } 1\leq i_1<\cdots <i_{n-1}<i_n, n\in \N\right\rbrace$$
is monochromatic.  
\end{thm}

Let $(S,\cdot)$ and $(T,\cdot)$ be two discrete semigroups, and $p\in \beta S$ and  $q\in \beta T$. Then the tensor product of $p$ and $q$ is defined as
\[
p\otimes q=\{A\subseteq S\times T: \{x\in S: \{y:(x,y)\in A\}\in q\}\in p\}
\]
where $x\in S$, $y\in T.$

Consider a groupoid operation $\star$ over $\beta\N$ defined as follows: first, consider the operation $f:\N^{2}\rightarrow \N$ defined as $f(n,m)=2^{n}m$. Let $\overline{f}:\beta\left(\N^{2}\right)\rightarrow\beta\N$ be its continuous extension. $\star$ is the restriction of $\overline{f}$ to tensor pairs, namely: for all $p,q\in\beta\N$
\[
p\star q=\overline{f}(p\otimes q)=2^{p}\odot q.
\]
The following theorem implies the Exponential version of the Hindman theorem.

\begin{thm}\cite{6, exp}\label{DiR} 
 Let $p\in E\left(K(\beta\mathbb{N},\oplus)\right)$
and $A\in p\star p$. For every $N\in\mathbb{N}$ and for every sequence
$\varPhi=\left(f_{n}:n\geq2\right)$ of functions $f_{n}:\mathbb{N}\rightarrow\mathbb{N},$
there exists a sequence $\left(a_{k}\right)_{k\in\mathbb{N}}$ such
that
\[
\mathcal{F}_{k,\varPhi}\left(a_{k}\right)_{k\in\mathbb{N}}=\left\{ a_{k}\cdot2^{\sum_{i=1}^{k-1}\lambda_{i}a_{i}}:k\in\mathbb{N},0\leq\lambda_{1}\leq N,\text{ and }0\leq\lambda_{i}\leq f_{i}\left(a_{i-1}\right)\text{ for }2\leq i\leq k-1\right\} \subset A.
\]
\end{thm}

For each $i\in \N$, choosing $f_i$ suitably, we have the exponential Hindman theorem.

 Let $N=1$, and choose the function $f_2(x)=2^x.$ Inductively assume that we have defined the function $f_n(x)$ for some $n\in \N_{>1}$. Now define $f_{n+1}(x)=(2^x)^{f_n(x)}$.  Now choosing $\varPhi=\left(f_{n}:n\geq2\right)$, Theorem \ref{DiR} immediately implies the existence of a monochromatic Hindman tower.

For every $p\in \mathbb{P},$ let $deg(p)$ be the degree of the polynomial $p,$ and let $coef(p)=\max \{|c|: c\text{ is a coefficient of }p\}.$ For $n\in \N_{>1},$ define a new operation $\star_n$ on $\N$ by $a\star_nb=n^ab.$
The following theorem was proved in \cite{exp}.

\begin{thm} \textup{\cite[Theorem 4.1.]{exp}} \label{poly}
Let $p\in E\left(K(\beta\mathbb{N},\oplus)\right)$
and $q\in \bN$ be such that for every $N\in \N$, each element of $q$ contains an $IP_N$ sets. Let $n\in \N_{>1},$  $A\in p\star_n q,$ and  $F_{1}\in \mathcal{P}_f(\mathbb{P})$. Let $\varPhi=\left(f_{n}:n\geq2\right)$ be a sequence
of functions $f_{n}:\mathbb{N}\rightarrow\mathbb{N}.$ 
For every $n(>1),x\in \N$, let $F_{n,x}=\{p\in \mathbb{P}:deg(p)\leq f_n(x)\text{ and } coef(p)\leq f_n(x)\}.$ 

Under this hypothesis we have a sequence $\left(x_{k}\right)_{k\in\mathbb{N}}$
such that

\[
\mathcal{PF}^n_{k,\varPhi}\left(x_{k}\right)_{k\in\mathbb{N}}=\left\{ x_{k}\cdot n^{\sum_{i=1}^{k-1}p_{i}\left(x_{i}\right)}:k\in\mathbb{N},p_{1}\in F_{1},\text{ and }p_{i}\in F_{i,f_{i}\left(x_{i-1}\right)}\text{ for }2\leq i\leq k-1\right\} \subset A.
\]
\end{thm}

Now we find two refinements of Theorem \ref{poly}.

\subsection{Proof of Theorem \ref{p1}}

The following theorem is the first improvement of Theorem \ref{poly}. As a corollary we have a monochromatic Hindman tower each of whose terms coming from  $MIP_r$ sets of a given length.

\begin{thm}\label{polynew} Let $n\in \N_{>1},$ $\langle N_i\rangle_{i\in \N}$ be a sequence in $\N$, $F_{1}\in \mathcal{P}_f(\mathbb{P})$, $p\in E(\overline{\left(K(\beta\mathbb{N},+)\right)})$ , $q\in \overline{E((\bN,+))}$ and  $A\in p\star_n q$. Let $\varPhi=\left(f_{n}:n\geq2\right)$ be a sequence
of functions $f_{n}:\mathbb{N}\rightarrow\mathbb{N}.$ 
For every $n \in \N_{>1}, G\in \mathcal{P}_f(\N)$, let $F_{n,G}=\{p\in \mathbb{P}:deg(p)\leq f_n(x)\text{ and } coef(p)\leq f_n(x) \text{ for }x=\max G\}.$ 
Under this hypothesis we have a sequence of finite sets $\langle G_i\rangle_{i\in \N}$ such that

\begin{enumerate}
    \item[(1)] for each $i\in \N,$ $|G_i|=N_i$, and

\item[(2)]  $$\mathcal{PF}^n_{k,\varPhi}\left(G_{k}\right)_{k\in\mathbb{N}}=\left\{ z_{k}\cdot n^{\sum_{i=1}^{k-1}p_{i}\left(y_{i}\right)}:k\in\mathbb{N},p_{1}\in F_{1},y_j\in FS(G_j)\cup FP(G_j)\text{ for }1\leq j\leq k-1, \right.$$
$$\,\,\,\,\,\,\,\,\,\,\,\,\,\,\,\,\,\,\,\,\,\,\,\,\,\,\,\,\,\,\,\,\,\,\,\,\,\,\,\,\,\,\,\,\,\,\,\,\,\,\,\,\,\,\,\,\,\,\,\,\,\,\,\,\,\,\,\,\,\,\,\,\,\,\,\,\,\,\,\,\,\,\,\,\,\,\,\,\,\,\,\,\,\,\,\,  z_k\in FS(G_k) \text{ and }\, p_{i}\in F_{i,f_{i}\left(G_{i-1}\right)}\text{ for }2\leq i\leq k-1 \Big\} \subset A.$$
\end{enumerate}
\end{thm}

\begin{proof}
To avoid the complexity in the calculation we prove this theorem up to $k=3.$

The rest of the part can be finished inductively. The technique is verbatim.

Let $p\in E\left(K(\beta\mathbb{N},+)\right)$, $q\in E((\bN,+))$ and
$$A\in p\star_n q=E_1(n,p)\cdot q.$$ Then
$$B=\left\{ m_0:\left\{ n_0:n^{m_0}n_0\in A\right\} \in q\right\} \in p.$$
Let $m_{1}\in B^{\star}$ and  $C_{1}= \left\{ n':n^{m_1}n'\in A\right\} \in q.$
Let $F_{1}\in \mathcal{P}_f(\mathbb{P})$.

 Hence there exists an $IP$ set $$FS(\langle n_{1,i} \rangle_{i\in \N})\subseteq C_{1}\in q.$$ 
For each $i\in \N,$ let $x_{1,i}=n^{m_{1}}n_{1,i}\in A$.  Now $FS(\langle x_{1,i}\rangle_{i\in \N} )\subset A.$ By Theorem \ref{ess1}, there
exists $a_{1}\in B^{\star}$, and $G_1\subset FS(\langle x_{1,i}\rangle_{i\in \N} )$ such that $|G_1|=N_1,$ and $$\Big\{ a_{1},a_{1}+p_{1}\left(y_{1}\right):p_{1}\in F_{1}, y_1\in FS(G_1)\cup FP(G_1)\Big\} \subset B^{\star}.$$

So, $$B_{2}=B^{\star}\cap\left(-a_{1}+B^{\star}\right)\cap\bigcap_{p_{1}\in F_{1}}\bigcap_{y_1\in FS(G_1)\cup FP(G_1)}-\left(a_{1}+p_{1}\left(y_{1}\right)\right)+B^{\star}\in p.$$

Let $$C_{2}=C_{1}\cap\bigcap_{p_{1}\in F_{1}}\bigcap_{y_1\in FS(G_1)\cup FP(G_1)}\left\{ n':n^{a_{1}+p_{1}\left(y_{1}\right)}n'\in A\right\} \cap\left\{ n':n^{a_{1}}n'\in A\right\} \in q. $$

Hence there exists an $IP$ set $$FS(\langle n_{2,i}\rangle_{i\in \N} )\subseteq C_{2}\in q.$$
For each $i\in \N,$ let $n^{a_{1}}n_{2,i}=x_{2,i}\in A$. Hence $FS(\langle x_{2,i}\rangle_{i\in \N})\subseteq A.$ And for each $y_1\in FS(G_1)\cup FP(G_1),$ and $y\in FS(\langle x_{2,i}\rangle_{i\in \N}),$ and $p_{1}\in F_{1}$, we have
$yn^{p_{1}\left(y_{1}\right)}\in A.$

As $FS(\langle x_{2,i}\rangle_{i\in \N})\subseteq A$ is an $IP$ set, from Theorem \ref{ess1}, we can choose a $G_2\subset FS(\langle x_{2,i}\rangle_{i\in \N})$ such that $|G_2|=N_2$ and

$$\left\{ a_{2},a_{2}+p_{2}\left(y_2\right):p_{2}\in F_{2,f_{2}\left(y_{1}\right)}, y_2\in FS(G_2)\cup FP(G_2)\right\} \subset B_{2}^{\star}.$$
Letting $c=a_{1}+a_{2}$, we have
\[
\left\{ c+p_{1}\left(y_{1}\right)+p_{2}\left(y_{2}\right):p_{1}\in F_{1},p_{2}\in F_{2,f_{2}\left(y_{1}\right)},y_i\in FS(G_i)\cup FP(G_i)\text{ for }1\leq i\leq 2\right\} \subset B^{\star}.
\]
 Let $$C_{3}=C_{1}\cap\bigcap_{p_{1}\in F_{1}}\bigcap_{p_{2}\in F_{2,f_{2}\left(x_{1}\right)}}\bigcap_{1\leq i\leq 2}\bigcap_{y_i\in FS(G_i)\cup FP(G_i)}\left\{ n':n^{c+p_{1}\left(y_{1}\right)+p_{2}\left(y_{2}\right)}n'\in A\right\} \in q$$
and $$B_{3}^{\star}=B^{\star}\cap\bigcap_{p_{1}\in F_{1}}\bigcap_{p_{2}\in F_{2,f_{2}\left(x_{1}\right)}}\bigcap_{1\leq i\leq 2}\bigcap_{y_i\in FS(G_i)\cup FP(G_i)}-(c+p_{1}\left(x_{1}\right)+p_{2}\left(x_{2}\right))+B^{\star}\in p.$$
Choose an $IP$ set $FS(\langle n_{3,i}\rangle_{i\in \N})\subseteq C_3\in q.$
 Again for each $i\in \N,$ let $n^{c}n_{3,i}=x_{3,i}\in A$. So $FS(\langle x_{3,i}\rangle_{i\in \N})\subseteq A.$ Hence for every $1\leq i\leq 2,$ and $y_i\in FS(G_i)\cup FP(G_i)$, $z\in FS(\langle x_{3,i}\rangle_{i\in \N})$, $p_{1}\in F_{1},\text{ and }p_{2}\in F_{2,f_{2}\left(x_{1}\right)}.$ we have
$zn^{p_{1}\left(y_{1}\right)+p_{2}\left(y_{2}\right)}\in A.$
As $FS(\langle x_{3,i}\rangle_{i\in \N})$ is an $IP$ set, and $B_3^*$ is piecewise syndetic set, we can use Theorem \ref{ess1} as before to find a set $G_3\subset FS(\langle x_{3,i}\rangle_{i\in \N})$ with $|G_3|=N_3.$

Now iterating this argument we have an infinite sequence of finite subsets $\langle G_{n}\rangle_{n\in\mathbb{N}}$
such that the desired result is true.
\end{proof}

Now we have the following corollary which shows that monochromatic Hindman's towers are abundant.

\begin{proof}[\textbf{Proof of Theorem \ref{p1}:}]
    
Let $I:\N\rightarrow \N$ be the polynomial defined by for all $x\in \N$, $I(x)=x.$ Let $f_2(x)=2^x,$ and define inductively $f_{n+1}(x)=(2^x)^{f_n(x)}$ for all $n\in \N$. In the above theorem if we choose these sequence of functions $(f_n)_{n\in \N}$, $F_{n,G}=\{I\}$ for all $G\in \mathcal{P}_f(\N),$ then we have a sequence of finite sets $\langle G_i\rangle_{i\in \N}$ such that for each $i\in \N,$ $G_i=\langle n^{x_{i,k}}\rangle_{k=1}^{N_i}$ such that the given pattern is monochromatic.
\end{proof}

\section*{Acknowledgement}  The author of this paper is supported by NBHM postdoctoral fellowship with reference no: 0204/27/(27)/2023/R \& D-II/11927.

\end{document}